\documentclass[12pt,a4paper,oneside]{amsart}
\usepackage[a4paper]{geometry}
\usepackage{mathptmx}

\usepackage{amssymb,enumerate,mathrsfs}

\DeclareMathOperator{\ad}{\mathsf{ad}}
\DeclareMathOperator{\C}{C}
\DeclareMathOperator{\Der}{Der}
\DeclareMathOperator{\gl}{\mathsf{gl}}
\DeclareMathOperator{\Homol}{H}
\DeclareMathOperator{\im}{Im}
\DeclareMathOperator{\Ker}{Ker}
\DeclareMathOperator{\N}{N}
\DeclareMathOperator{\pgl}{\mathsf{pgl}}
\DeclareMathOperator{\psl}{\mathsf{psl}}
\DeclareMathOperator{\Rad}{Rad}
\DeclareMathOperator{\rk}{rk}
\DeclareMathOperator{\sym}{S}
\DeclareMathOperator{\Sl}{\mathsf{sl}}
\DeclareMathOperator{\su}{\mathsf{su}}
\DeclareMathOperator{\tr}{tr}
\DeclareMathOperator{\Z}{Z}

\newcommand{\vertbar}{\>|\>}
\newcommand{\set}[2]{\ensuremath{\{ #1 \vertbar #2 \}}}
\def\aform     {\ensuremath{\langle\,\cdot\, , \cdot\,\rangle}}

\newtheorem{theorem}{Theorem}
\newtheorem*{A}{Theorem A}
\newtheorem{lemma}{Lemma}
\newtheorem{corollary}{Corollary}
\newtheorem*{conjecture}{Conjecture}

\begin{document}

\title{On regular Lie algebras}

\author{Pasha Zusmanovich}
\address{University of Ostrava, Czech Republic}
\email{pasha.zusmanovich@osu.cz}

\date{First written June 20, 2020; last minor revision November 12, 2022}
\thanks{Comm. Algebra \textbf{49} (2021), no.3, 1104--1119}

\begin{abstract}
We study so called regular Lie algebras, i.e. Lie algebras in which each nonzero
element is regular. We make a connection with an open problem whether any 
element of reduced trace zero in a simple associative algebra is a commutator.
\end{abstract}


\maketitle

\section*{Introduction}

A Lie algebra is called regular, if for any its nonzero element $x$, the 
characteristic polynomial of the adjoint linear map $\ad x$ has the same 
minimal possible non-vanishing power. The study of regular Lie algebras was 
initiated in \cite{mna}. There we proved some elementary properties of those 
algebras over a field of characteristic zero, and established their connection 
with another classes of Lie algebras, anisotropic and minimal nonabelian. Here,
after setting the necessary definitions and notation, and recalling the 
necessary facts in the preliminary \S \ref{s-not}, we continue to study regular
Lie algebras, by extending results of \cite{mna} to include the case of positive 
characteristic; this is done in \S \ref{sec-regular}. The brief \S \ref{s-real}
contains a description of the real case. In \S \ref{s-min} we discuss a 
connection with minimal non-abelian Lie algebras and minimal non-commutative 
associative algebras, and describe minimal non-regular algebras. In 
\S \ref{s-ass} we discuss an interesting open problem: whether any element of 
reduced trace zero in a simple associative algebra can be represented as a 
commutator of two elements, and suggest an attack on (a particular case of) this problem utilizing regular Lie algebras.

\section{Notation, definitions, and recollections}\label{s-not}

\renewcommand{\thesubsection}{\arabic{subsection}}

\subsection{Ground field, general notation}

Throughout this note, all algebras are assumed to be finite-di\-men\-si\-o\-nal,
defined over an infinite field $K$. In most of our results, we stipulate 
additional conditions on $K$ (like being perfect, or not of small 
characteristic), and these conditions are always specified explicitly in the 
statements of theorems and lemmas. By $\overline K$ we denote the algebraic 
closure of $K$. Needless to say, all that will follow is nontrivial only if the
ground field $K$ is not algebraically closed (see the beginning of 
\S \ref{s-real}).

If $S$ is a subalgebra of a Lie algebra $L$, the normalizer of $S$ in $L$, i.e.
the subalgebra $\set{x \in L}{[S,x] \subseteq S}$, is denoted by $\N_L(S)$. The
centralizer of an element $x \in L$, i.e. the subalgebra 
$\set{y \in L}{[y,x] = 0}$, is denoted by $\C_L(x)$. By $[L,L]$ and $\Z(L)$ are 
denoted the commutant and the center of $L$, respectively. By $\Der(L)$, or, 
occasionally, by $\Der_K(L)$ if we want to stress over which ground field $K$ we
are working at the moment, we denote the Lie algebra of derivations of an 
algebra $L$. If $A$ is an associative algebra, by $A^{(-)}$ we will denote its
``minus'' algebra, i.e., the Lie algebra defined on the same vector space $A$
subject to multiplication $[a,b] = ab - ba$. The direct sum of algebras is 
denoted by $\oplus$, while the direct sum of vector spaces (usually also 
subalgebras, but with not necessary trivial multiplication between them), is 
denoted by $\dotplus$. 

Recall that a symmetric bilinear form $\aform: L \times L \to K$ on a
Lie algebra $L$ is called \emph{invariant}, if 
$\langle [x,y],z \rangle = \langle x,[y,z] \rangle$ for any $x,y,z \in L$. The 
orthogonal complement of a subspace $S$ of $L$ with respect to $\aform$ is 
denoted by $S^\perp$.

\subsection{Fitting decomposition}

Following \cite[Chapter VII, \S1]{bourbaki}, we will call the set $X$ of 
elements of a Lie algebra $L$ \emph{almost commuting}, if for any $x,y \in X$, 
there is an integer $n$ such that $(\ad x)^n y = 0$. This is, essentially, 
equivalent to the existence of the Fitting decomposition of $L$ with respect to the adjoint action of $X$:
$$
L = L^0(X) \dotplus L^1(X) ,
$$
where the Fitting $0$-component is defined as
$$
L^0(X) = \set{y \in L}
{\text{for any } x\in X, \text{ there is integer } n \text{ such that } 
(\ad x)^n (y) = 0} ,
$$
and the Fitting $1$-component is defined as
$$
L^1(X) = \sum_{x \in X} \big(\bigcap_{n \ge 0} (\ad x)^n (L)\big) .
$$

\subsection{Regular Lie algebras}

Let $L$ be a Lie algebra. For any $x \in L$, consider the characteristic 
polynomial of the adjoint map $\ad x$:
$$
\chi_{\ad x}(t) = \det (t - \ad x) = 
\sum_{i=0}^{\dim L} a_i(x) t^i .
$$
Recall (\cite[Chapter VII, \S 2.2]{bourbaki} or 
\cite[Chapter III, \S 1]{jacobson}) that the rank of $L$, denoted by $\rk L$, is
the minimal number $r$ such that $a_r(x) \ne 0$ for some $x \in L$ (or, what is
the same, the multiplicity of the eigenvalue zero of $\ad x$). An element $x$ is
called \emph{regular}, if this minimum is attained for it, i.e., if 
$a_{\rk L}(x) \ne 0$. An element $x \in L$ is regular if and only if the Fitting
$0$-component $L^0(x)$ is of the minimal possible dimension $\rk L$.

A Lie algebra is called \emph{regular}, if each its nonzero element is regular. 
(Note that this notion is entirely different from the notion of regular 
subalgebra in the classical works of Dynkin about subalgebras of simple Lie 
algebras over algebraically closed fields of characteristic zero).

\subsection{Anisotropic algebras}

Recall that a Lie algebra $L$ is called \emph{anisotropic}, if for any $x\in L$,
$\ad x$ is a semisimple linear map. Anisotropic Lie algebras were studied in 
\cite{farn-anis} and \cite{premet}--\cite{premet-cartan} (see also 
\cite[\S 1]{mna}). We will use repeatedly the following powerful result due to 
Premet.

\begin{A} (The ground field $K$ is perfect of characteristic $\ne 2,3$).
If $L$ is an anisotropic centerless Lie algebra, then $L$ is sandwiched between
the direct sum of forms of classical simple Lie algebras, and their derivation 
algebras. That is,
\begin{equation}\label{eq-s}
S_1 \oplus \dots \oplus S_n \subseteq L \subseteq 
\Der(S_1) \oplus \dots \oplus \Der(S_n)
\end{equation}
for some forms of classical simple Lie algebras $S_1, \dots, S_n$.
\end{A}
Here and below, by abuse of notation, we identify the algebra $S_i$ with the 
algebra of its inner derivations $\ad(S_i)$.

In characteristic zero the statement is trivial, in characteristic $p>5$ it
as established in \cite[Corollary 2]{premet}, and in characteristic $p=5$, in
\cite{premet-p5}.

We also need to consider an auxiliary, a priori slightly more general condition,
namely, that in a Lie algebra any nilpotent element is central. Such Lie 
algebras will be dubbed \emph{nilpotent-free}. It is obvious that any 
anisotropic Lie algebra is nilpotent-free, and below we will see that over 
perfect fields these two notions coincide (Lemma \ref{lemma-anis}).

\subsection{Minimal non-$\mathscr P$ algebras}

Let $\mathscr P$ be a property of Lie algebras. A Lie algebra is called 
\emph{minimal non-$\mathscr P$}, if it does not satisfy $\mathscr P$, but any
its proper subalgebra satisfies $\mathscr P$. We will encounter 
\emph{minimal non-abelian}, \emph{minimal non-nilpotent}, and 
\emph{minimal non-regular} Lie algebras. The first two classes were studied
extensively in the literature, see \cite{farn-anis}, \cite{gein},
\cite{gein-book}, \cite{muller}, \cite{premet-cartan}, \cite{stitzinger},
\cite{towers-laa}, and \cite{mna}, among others; minimal non-regular Lie 
algebras will be discussed in \S \ref{s-min}.

The same notion can be considered for other classes of algebras. Here we will 
encounter \emph{minimal non-commutative} associative algebras; such algebras 
were studied in \cite{gein}.

\subsection{Field extensions, centroid, and central algebras}\label{ss-cent}

If $L$ is a Lie algebra over $K$, $K \subset F$ is a field extension, and $L$ 
has also an $F$-vector space structure, compatible with $K$, we can consider the
Lie algebra $L_F$ over $F$, called \emph{restriction over $F$} (not to be 
confused with the Lie algebra $L \otimes_K F$ obtained by \emph{extension} of 
the ground field from $K$ to $F$). It is easy to see that the $K$-algebra $L$ is
regular (respectively, anisotropic) if and only if the $F$-algebra $L_F$ is 
regular (respectively, anisotropic).

By $\Omega(L)$ we denote the centroid of an algebra $L$. An algebra $L$ is 
called \emph{central}, if $\Omega(L) = K$. For a simple algebra $L$, this is 
equivalent to the condition that the $F$-algebra $L \otimes_K F$ remains to be 
simple under any extension $F$ of the ground field $K$.

If $L$ is simple, then its centroid $\Omega(L)$ is a field, an extension of the
ground field $K$, and we can consider the algebra $L_{\Omega(L)}$ defined over 
$\Omega(L)$. The latter algebra is a central $\Omega(L)$-algebra.

If $L$ is simple and the ground field $K$ is perfect, then being extended to
a sufficiently large field $F$ (for example, $\Omega(L)$, or the algebraic 
closure $\overline K$), $L$ is decomposed into the direct sum of a number of
isomorphic copies of a central simple algebra, i.e. there is an isomorphism of 
$F$-algebras
$$
L \otimes_K F \simeq \overline S_1 \oplus \dots \oplus \overline S_n ,
$$ 
where $\overline S_i \simeq \overline S_j$ are central simple Lie algebras over 
$F$ (see, for example, \cite[\S 2]{eld}). Of course, if $L$ is central over $K$,
then $n=1$.

\subsection{Almost simple and classical algebras}\label{ss-as}

Recall that a Lie algebra $L$ is called \emph{almost simple}, if there is a 
simple subalgebra $S$ of $L$ such that $L$ is sandwiched between $S$ and its
derivation algebra $\Der(S)$:
\begin{equation}\label{eq-as}
S \subseteq L \subseteq \Der(S) .
\end{equation}

The simple algebra $S$ is called the \emph{simple constituent} of the almost
simple Lie algebra $L$. An almost simple Lie algebra is called \emph{classical},
if its simple constituent is classical.

If $S$ is central simple and classical, then we have $\Der(S) \simeq S$ in all 
the cases except when characteristic of the ground field $p>0$, and $S$ is of 
type $A_{kp-1}$ for some $k$. To see what is happening in these exceptional 
cases, assume that $S$ is split (which always can be achieved by passing to the
algebraic closure of the ground field).

Recall that the algebra of $n \times n$ matrices has a peculiarity if $n$ 
divides $p$; say, $n=kp$. In this case, the $kp \times kp$ identity matrix $E$ 
has trace zero, so the Lie algebra $\Sl_{kp}(K)$ of traceless matrices is no 
longer simple, but has the one-dimensional center linearly spanned by $E$. The 
quotient algebra 
$$
\psl_{kp}(K) = \Sl_{kp}(K)/KE
$$
is simple. Similarly, we can consider the quotient algebra 
$$
\pgl_{kp}(K) = \gl_{kp}(K)/KE
$$
which is semisimple and contains $\psl_{kp}(K)$ as an ideal of codimension $1$.
Moreover, 
$$
[\pgl_{kp}(K), \pgl_{kp}(K)] = \psl_{kp}(K) ,
$$
and 
$$
\Der(\psl_{kp}(K)) \simeq \pgl_{kp}(K)
$$
(see \cite[Chapter V, \S 5]{seligman} or \cite[p.~152]{premet}), so 
$\pgl_{kp}(K)$ is almost simple.

\subsection{Forms (of algebras over fields)}\label{ss-forms}

We will need a description of forms of simple classical Lie algebras of type 
$A$, i.e., algebras of the kind $\Sl_n(K)$ and $\psl_{kp}(K)$ (see, for example,
\cite[Chapter IV, \S 3]{seligman}). Any central simple form of 
$\Sl_n(\overline K)$ is either of the kind $[A^{(-)},A^{(-)}]$, where $A$ is a 
central simple associative algebra, or of the kind $[\sym^-(A,J),\sym^-(A,J)]$,
where $A$ is a simple associative algebra whose center is a quadratic extension
of $K$, $J$ is involution on $A$ of the second kind, and 
$\sym^-(A,J) = \set{a \in A}{J(a) = -a}$ is the Lie algebra of 
$J$-skew-symmetric elements of $A$. In both cases $A$ is of degree $n$ over its
center. If we are interested in forms of $\psl_{kp}(\overline K)$, we should 
everywhere take the quotient by the central ideal $K1$, where $1$ is the unit in the respective associative algebra $A$, getting in this way 
Lie algebras of the kind $[A^{(-)},A^{(-)}]/K1$ and 
$[\sym^-(A,J),\sym^-(A,J)]/K1$, respectively.

\section{Structure and properties of regular Lie algebras}\label{sec-regular}

The next three lemmas are elementary observations from \cite[\S 2]{mna} readily 
following from the properties of regular elements, as expounded, for example, in
\cite[Chapter VII, \S 2.2]{bourbaki}, and are valid in any characteristic.

\begin{lemma}\label{lemma-1}
Any nilpotent Lie algebra is regular.
\end{lemma}

\begin{lemma}\label{lemma-2}
A subalgebra of a regular Lie algebra is regular.
\end{lemma}

\begin{lemma}\label{lemma-3}
A non-semisimple regular Lie algebra is nilpotent. 
\end{lemma}

\begin{lemma}\label{lemma-4}
If $I$ is a nontrivial ideal of a regular Lie algebra $L$, then $L/I$ is 
nilpotent.
\end{lemma}

\begin{proof}
First note that $L/I$ is a regular Lie algebra (as follows, for example, from
\cite[Chapter VII, \S 2.2, Proposition 8]{bourbaki}).

Since for any element $x\in L$, the ideal $I$ is an $\ad x$-invariant subspace 
of $L$, we have
\begin{equation}\label{eq-chi}
\chi_{\ad x}(t) = \chi_{\ad_I x}(t) \, \chi_{\ad_{L/I} (\overline x)}(t) ,
\end{equation}
where $\ad_I x$ is the restriction of $\ad x$ to $I$, $\overline x$ is the image
of $x$ under the canonical homomorphism $L \to L/I$, and $\ad_{L/I} (\overline x)$
is its adjoint map acting on the algebra $L/I$.

Picking $x \notin I$, we get from (\ref{eq-chi}) 
$$
\rk L = \rk I + \rk L/I .
$$ 
On the other hand, if $x \in I$, $x \ne 0$, then 
$\overline x = 0$, $\chi_{\ad_{L/I} (\overline x)}(t) = t^{\dim L/I}$, and 
(\ref{eq-chi}) implies
$$
\rk L = \rk I + \dim L/I .
$$
Consequently, $\rk L/I = \dim L/I$, and hence $L/I$ is nilpotent.
\end{proof}

\begin{lemma}\label{lemma-almost}
A semisimple regular Lie algebra is almost simple.
\end{lemma}

\begin{proof}
Let $L$ be a semisimple regular Lie algebra. In characteristic zero, $L$ is the
direct sum of simple algebras, and by Lemma \ref{lemma-4}, $L$ is simple.

In characteristic $p>0$, by the Block theorem about the structure of semisimple 
Lie algebras,
\begin{equation}\label{eq-block}
\bigoplus_{i\in \mathbb I} \big(S_i \otimes O_{n_i}\big) \subseteq L \subseteq 
\bigoplus_{i\in \mathbb I} 
\big(\Der(S_i) \otimes O_{n_i} \dotplus \Der(O_{n_i})\big)
\end{equation}
for some finite set of simple Lie algebras $\{S_i\}_{i\in \mathbb I}$, and some 
nonnegative integers $\{n_i\}_{i\in \mathbb I}$. Here
$$
O_n = K[x_1,\dots,x_n]/(x_1^p, \dots, x_n^p)
$$
denotes the reduced polynomial algebra in $n$ variables. By Lemma \ref{lemma-2},
the direct sum $\bigoplus_{i\in \mathbb I} (S_i \otimes O_{n_i})$ is regular, 
and by Lemma \ref{lemma-4}, it consists of a just one summand, i.e., 
$|\mathbb I| = 1$ and
$$
S \otimes O_n \subseteq L \subseteq \Der(S) \otimes O_n \dotplus \Der(O_n)
$$
for suitable $S$ and $n$. Again, $S \otimes O_n$ is regular, and by 
Lemma \ref{lemma-3} we have $n=0$, i.e., $O_n = K$ and (\ref{eq-as}) holds.
\end{proof}

Now we will establish a connection of regular Lie algebras with another class
of Lie algebras, namely, anisotropic Lie algebras. 

\begin{lemma}\label{lemma-anis-sub}
A subalgebra of an anisotropic Lie algebra is anisotropic.
\end{lemma}

\begin{proof}
Let $S$ be a subalgebra in anisotropic Lie algebra $L$. For any $x \in S$, the
restriction of the semisimple linear map $\ad x: L \to L$, to the invariant 
subspace $S$ is also semisimple, thus $x$ is semisimple element in $S$.
\end{proof}

\begin{lemma}\label{lemma-solv}
A solvable nilpotent-free Lie algebra is abelian.
\end{lemma}

\begin{proof}
See the proof of \cite[Proposition 1.2]{farn-anis} (where a formally weaker 
assertion is stated -- that a solvable anisotropic Lie algebra is abelian).
\end{proof}

\begin{lemma}\label{lemma-id}
An ideal of a nilpotent-free Lie algebra is nilpotent-free.
\end{lemma}

\begin{proof}
Let $I$ be an ideal of a nilpotent-free Lie algebra $L$. If $x \in I$ such that
$(\ad x)^n (I) = 0$ for some $n$, then
$$
(\ad x)^{n+1}(L) = (\ad x)^n [L,x] \subseteq (\ad x)^n(I) = 0 ,
$$
and hence $\ad x = 0$ .
\end{proof}

\begin{lemma}\label{lemma-anis} (The ground field $K$ is perfect).
A Lie algebra is anisotropic if and only if it is nilpotent-free.
\end{lemma}

\begin{proof}
The ``only if'' part is obvious, so let us prove the ``if'' part. Let $L$ be a
nilpotent-free Lie algebra. In characteristic zero, Lemma \ref{lemma-id} implies
that the radical $\Rad(L)$ of $L$ is nilpotent-free, and Lemma \ref{lemma-solv}
implies that $\Rad(L)$ is abelian. Then for any $x \in \Rad(L)$, 
$(\ad x)^2 = 0$, hence $x$ is central and $L$ is either reductive or abelian.
But a reductive Lie algebra contains semisimple and nilpotent parts in the 
Jordan--Chevalley decomposition of each its element (see, for example, 
\cite[Chapter VII, \S 5.1, Proposition 2]{bourbaki}). Thus any element in $L$ is
a sum of a semisimple and a central element, and hence is semisimple.

In the case of positive characteristic, consider the universal $p$-envelope 
$R(L)$ of $L$ (\cite[\S 1]{11} or \cite[Chapter 2, \S 5]{strade-farn}). Let 
$x \in R(L)$ be a nilpotent element. We may assume then that $x^{[p]^n} = 0$ for
some $n$. As there is $y \in L$ such that $x = y^{[p]^k}$ for some $k$, we have 
$y^{[p]^{n+k}} = 0$, thus $y$ is a nilpotent element of $L$, hence $y$ belongs 
to the center of $L$, and $x = y^{[p]^k}$ belongs to the center of $R(L)$. Consequently, $R(L)$ is
nilpotent-free. Since $R(L)$ is a Lie $p$-algebra, any its element decomposes
into the sum of a semisimple and a $p$-nilpotent (and hence nilpotent) element
(the so called Jordan--Chevalley--Seligman decomposition, see 
\cite[Theorem V.7.2]{seligman} or \cite[Chapter 2, Theorem 3.5]{strade-farn}). 
Then, as in the case of zero characteristic, $R(L)$ is anisotropic. By 
Lemma~\ref{lemma-anis-sub}, $L$ is anisotropic.
\end{proof}

\begin{lemma}\label{lemma-nn} (The ground field $K$ is perfect).
A non-nilpotent regular Lie algebra is anisotropic.
\end{lemma}

\begin{proof}
A non-nilpotent regular Lie algebra is obviously nilpotent-free (in fact, any
nilpotent element vanishes), and by Lemma \ref{lemma-anis} is anisotropic.
\end{proof}

\begin{theorem}\label{th-as} 
(The ground field $K$ is perfect of characteristic $p \ne 2,3$).
A non-nilpotent regular Lie algebra $L$ is a form of an almost simple classical
Lie algebra. Moreover, if the simple constituent of $L$ is central simple, then
$L$ is a form of either a simple classical Lie algebra, or of an algebra of the
kind $\pgl_{kp}(\overline K)$.
\end{theorem}

\begin{proof}
By Lemma \ref{lemma-3} and Lemma \ref{lemma-almost}, $L$ is almost simple. In 
characteristic zero, almost simplicity is equivalent to simplicity, and any 
simple Lie algebra is a form of a classical one.

In the case of characteristic $p>3$, by Lemma \ref{lemma-nn}, $L$ is 
anisotropic. Apply Theorem A. Since $L$ is almost simple, in the inclusion
(\ref{eq-s}) we have $n=1$. Putting $S = S_1$ and passing to the algebraic 
closure of the ground field, we have
$$
S \otimes_K \overline K \subseteq L \otimes_K \overline K \subseteq 
\Der(S \otimes_K \overline K) .
$$

If $S$ is central, then $S \otimes_K \overline K$ is simple. The only simple 
classical Lie algebras which have outer derivations, are Lie algebras 
$\psl_{kp}(\overline K)$ for some nonnegative integer $k$. Hence either 
$L \otimes_K \overline K = S \otimes_K \overline K$, or 
$S \otimes_K \overline K \simeq \psl_{kp}(\overline K)$. In the latter case,
since $\psl_{kp}(\overline K)$ is of codimension one in 
$\Der(\psl_{kp}(\overline K)) \simeq \pgl_{kp}(\overline K)$, we have that 
$L \otimes_K \overline K$ is isomorphic to either $\psl_{kp}(\overline K)$, or 
to $\pgl_{kp}(\overline K)$.
\end{proof}

The following was proved as \cite[Theorem 3]{mna} in the zero characteristic 
case. Comparing the proof below with the proof of \cite[Theorem 3]{mna}, the 
reader can appreciate how much more efforts are required to handle the case of
positive characteristic.

\begin{theorem}\label{th-subalg} 
(The ground field $K$ is perfect of characteristic $\ne 2,3$).
Let $L$ be a non-solvable Lie algebra with the trivial center. Then the 
following are equivalent:
\begin{enumerate}[\upshape(i)]
\item $L$ is regular;
\item any proper subalgebra of $L$ is regular;
\item any proper subalgebra of $L$ is either almost simple, or abelian.
\end{enumerate}
\end{theorem}

\begin{proof}
(i) $\Rightarrow$ (ii) follows from Lemma \ref{lemma-2}.

(ii) $\Rightarrow$ (iii): 
By Lemma \ref{lemma-3}, Lemma \ref{lemma-almost}, and
Lemma \ref{lemma-nn}, any proper subalgebra of $L$ is either almost simple 
anisotropic, or nilpotent. Our goal is to prove that any nilpotent subalgebra is
abelian. We consider multiple possible cases.

\emph{Case 1}. $L$ is not semisimple. Let $I$ be a nontrivial abelian ideal of 
$L$.

\emph{Case 1.1}. $L$ contains a proper almost simple subalgebra $S$. Since any 
element of $I$ is nilpotent, we have $S \cap I = 0$. Further, $S \dotplus I$ is
a subalgebra of $L$, which is not regular, and hence coincides with the whole 
$L$. We have $[L,L] = [S,S] \dotplus [S,I]$.

\emph{Case 1.1.1}. $[L,L] = L$. Then $[S,S] = S$, i.e., $S$ is perfect, and 
$[S,I] = I$. Consider the action of $S$ on $I$ as a representation $\rho$ of 
$S$. For any $x \in S$, the subalgebra $Kx \dotplus I$ is a proper subalgebra of
$L$ which is not almost simple anisotropic, hence it is nilpotent, and $x$ acts
on $I$ nilpotently. By the Engel theorem, $S / \Ker \rho$ is nilpotent. On the other 
hand, $S / \Ker \rho$ is a homomorphic image of a perfect Lie algebra, hence is
also perfect. Consequently, $S / \Ker \rho = 0$, i.e., $[S,I] = 0$, and $I$ lies
in the center of $L$ and hence vanishes, a contradiction.

\emph{Case 1.1.2}. $[L,L]$ is a proper subalgebra of $L$. It is not nilpotent, 
hence anisotropic, and thus $[S,I] = 0$, i.e., $I$ lies in the center of $L$, a
contradiction.

\emph{Case 1.2}. All proper subalgebras of $L$ are nilpotent. Then, according to
\cite[Theorem 8.8, part 3]{gein-book}, all proper subalgebras of $L$ are 
abelian.

\emph{Case 2}. $L$ is semisimple. Repeating the reasoning in the proof of 
Lemma \ref{lemma-almost} based on the Block theorem, we get that $L$ is almost
simple: $S \subseteq L \subseteq \Der(S)$ for a simple Lie algebra $S$.

\emph{Case 2.1}. $S$ is a proper subalgebra of $L$. Then $S$ is regular, and by
Theorem \ref{th-as} is an anisotropic form of a classical Lie algebra. Take an
arbitrary nonzero element $d \in \Der(S)$ such that $d \in L$ and $d \notin S$.
Now we use a reasoning from the proof of \cite[Theorem 4.1]{farn-anis}, suitably
modified for our purpose. 

Let us prove by induction that $\Ker(d^n)$ is an abelian subalgebra of $S$
for any $n$. Indeed, if $n=1$, we have that $\Ker(d)$ is a proper subalgebra of
$S$, and hence is either abelian, or almost simple. On the other hand, 
$\Ker(d) \dotplus Kd$ is a subalgebra of $L$. Since $d$ is a central element in
this subalgebra, it is a proper subalgebra, and hence is regular. Since it 
contains a nonzero central element, it is nilpotent, and hence $\Ker(d)$ cannot
be almost simple. This shows that $\Ker(d)$ is abelian.

Now suppose that $\Ker(d^n)$ is an abelian subalgebra of $S$. For any 
$x,y \in S$, we have
\begin{equation}\label{eq-dd}
d^{n+1}([x,y]) = \sum_{i=0}^{n+1} [d^{n+1-i}(x),d^i(y)] .
\end{equation}

If $x,y \in \Ker(d^{n+1})$, then all terms at the right-hand side of this 
equality vanish, what shows that $\Ker(d^{n+1})$ is a subalgebra of $S$. On the
other hand, similarly with the case $n=1$, $\Ker(d^{n+1}) \dotplus Kd$ is a 
subalgebra of $L$. 

\emph{Case 2.1.1}. $\Ker(d^{n+1}) \dotplus Kd$ is a proper subalgebra of $L$. 
Then it is regular, and, since $d$ is a nilpotent element in it, it is 
nilpotent. Thus $\Ker(d^{n+1})$ is a nilpotent subalgebra of $S$, and by 
Lemmas~\ref{lemma-anis-sub} and \ref{lemma-solv}, $\Ker(d^{n+1})$ is abelian.

\emph{Case 2.1.2}. $L = \Ker(d^{n+1}) \dotplus Kd$. Then $S = \Ker(d^{n+1})$, 
and both $\Ker(d)$ and $\im(d)$ lie in $\Ker(d^n)$. Consequently, 
\begin{equation}\label{eq-d}
\dim S = \dim \Ker(d) + \dim \im(d) \le 2 \dim \Ker(d^n) .
\end{equation}

Let $H$ be a maximal abelian subalgebra of $S$ containing $\Ker(d^n)$. Since $H$
is an ideal in its normalizer $\N_S(H)$, the latter is a proper abelian 
subalgebra of $S$, and hence $N_S(H) = H$, i.e. $H$ is a Cartan subalgebra. 
Since $S$ is a form of a classical simple Lie algebra, all its Cartan 
subalgebras are of the same dimension, equal to $\rk S$. Now from (\ref{eq-d}) 
we have $\dim S \le 2\rk S$. Passing to the algebraic closure of the ground 
field, and owing to the fact that the rank is preserved under field extensions,
we have the same inequality $\dim \overline S \le 2\rk \overline S$ for a 
semisimple classical Lie algebra $\overline S = S \otimes_K \overline K$ over an
algebraically closed field $\overline K$, which is, obviously, impossible.

Since $\Ker(d^n)$ is abelian for any $n$, $d$ is not nilpotent. Since any 
nonzero element of $S$ is not nilpotent too, we get that $L$ is nilpotent-free,
and by Lemma \ref{lemma-anis}, $L$ is anisotropic, and then by 
Lemma~\ref{lemma-solv}, any of its nilpotent subalgebras are abelian. 

\emph{Case 2.2}. $L = S$, i.e., $L$ is simple. Any proper nilpotent subalgebra 
of $L$ is contained in a maximal nilpotent subalgebra $H$. The latter is either 
maximal among all (not just nilpotent) subalgebras, or contained in an 
anisotropic almost simple subalgebra, and by Lemma \ref{lemma-solv} is abelian.
If $H$ is maximal, then its normalizer $\N_L(H)$ is a subalgebra of $L$ 
containing $H$, then either $\N_L(H) = L$ in which case $H$ is an ideal of $L$,
a contradiction, or $\N_L(H) = H$, in which case $H$ is a Cartan subalgebra of 
$L$. To summarize: any nilpotent subalgebra of $L$ is either abelian, or 
contained in a Cartan subalgebra of $L$.

By \cite[Corollary]{premet-bssr}, any simple Lie algebra is either a form of a 
classical Lie algebra, or possesses absolute zero divisors, i.e., nonzero 
elements $c \in L$ such that $(\ad c)^2 = 0$.

\emph{Case 2.2.1}. $L$ is a form of a classical simple Lie algebra. Then any 
Cartan subalgebra of $L$, including $H$, is abelian.

\emph{Case 2.2.2}. $L$ possesses an absolute zero divisor $c \in L$. The 
centralizer $\C_L(c)$ is a subalgebra of $L$ with a central element $c$, hence 
it cannot be almost simple, and is nilpotent. 

\emph{Case 2.2.2.1}. $\C_L(c)$ is abelian. Let us prove by induction that
\begin{equation}\label{eq-n}
[\dots[c,\underbrace{L],\dots,L}_{n \text{ times }}] \subseteq \C_L(c)
\end{equation}
for any $n$. For $n=1$ this is true: $[c,L] \subseteq \C_L(c)$. Suppose the 
inclusion (\ref{eq-n}) holds for some $n>1$. By the Jacobi identity,
$$
[[[\dots[c,\underbrace{L],\dots,L}_{n \text{ times }}],L],c] \subseteq
[[[\dots[c,\underbrace{L],\dots,L}_{n \text{ times }}],c],L] \>+\>
[[\dots[c,\underbrace{L],\dots,L}_{n \text{ times }}],[c,L]] .
$$
By the induction assumption, the first summand at the right-hand side of this
inclusion vanishes, and the second summand is the product of two elements from 
$\C_L(c)$, and thus vanishes too. Thus the left-hand side vanishes, what is
equivalent to
$$
[\dots[c,\underbrace{L],\dots,L}_{n+1 \text{ times }}] \subseteq \C_L(c) .
$$

The sum of terms at the left-hand side of the inclusion (\ref{eq-n}) for all $n$
is the ideal of $L$ generated by $c$, and thus coincides with $L$. Consequently,
$L \subseteq \C_L(c)$, a contradiction.

\emph{Case 2.2.2.2}. $\C_L(c)$ is contained in a Cartan subalgebra $H$. Consider
the Fitting decomposition of $L$ with respect to $H$: $L = H \dotplus L_1$. 
Since $c \in H$, we have 
$$
[H,c] \dotplus [L_1,c] = [L,c] \subseteq \C_L(c) \subseteq H ,
$$
what implies $[L_1,c] = 0$ and $[L,c] = [H,c]$. The latter equality means that 
for any $x \in L$ there is $h \in H$ such that $[x,c] = [h,c]$, therefore 
$x - h \in \C_L(c)$, and $x \in H + \C_L(c) = H$, a contradiction.

\medskip

(iii) $\Rightarrow$ (i):

\emph{Case 1}. $L$ is not semisimple. Let $I$ be a nontrivial abelian ideal of
$L$. Take $x \in L$, $x \notin I$, and consider a subalgebra $I \dotplus Kx$. 
Obviously, this is a proper subalgebra of $L$ which is not almost simple, and
hence is abelian. This shows that $I$ contained in the center of $L$, and thus
$I=0$, a contradiction.

\emph{Case 2.} $L$ is semisimple. The reasonings based on the Block theorem, the
same as in the proof of Lemma \ref{lemma-almost}, show that $L$ is almost 
simple. Fix a nonzero $x \in L$, and let us prove by induction that 
$\Ker\big((\ad x)^n\big)$ is an abelian subalgebra of $L$ for any $n$. The proof
is very similar to the proof of the analogous statement in Case 2.1 of the 
implication (ii) $\Rightarrow$ (iii) above.

For $n=1$, $\Ker(\ad x)$ is a subalgebra of $L$ with a central element $x$, 
hence it cannot be almost simple, and thus is abelian. Now suppose 
$\Ker\big((\ad x)^n\big)$ is an abelian subalgebra of $L$ for some $n$. Using
the generalized Leibniz formula (\ref{eq-dd}) for the case $d = \ad x$, we get 
that $\Ker\big((\ad x)^{n+1}\big)$ is a subalgebra of $L$. Again, this 
subalgebra has a central element $x$, hence cannot be almost simple, and thus is
abelian. 

Therefore, $\ad x$ cannot be nilpotent, thus $L$ is nilpotent-free, and by 
Lemma \ref{lemma-anis} is anisotropic. By Theorem A, the simple constituent $S$
of $L$ is a form of a classical simple Lie algebra (and, in particular, all 
Cartan subalgebras of $S$ are of the same dimension $\rk S$). Now the same 
reasoning as in the proof of the implication (iii) $\Rightarrow$ (i) in 
\cite[Theorem 3]{mna} shows that $S$ is regular. If $L = S$ we are done, so 
assume $S \subsetneqq L \subseteq \Der(S)$. We can write
\begin{equation}\label{eq-sd}
L = S \dotplus D ,
\end{equation}
where $D$ is a subspace of $\Der(S)$, consisting of outer derivations of $S$. 
According to \S \ref{s-not}.\ref{ss-as}, we have 
$[\Der(S),\Der(S)] \subseteq S$, thus $D$ is an abelian subalgebra of $L$.

For any $x \in S$, we have
$$
\C_L(x) = \set{y + d}{y \in S, d \in D, [y,x] = d(x)} .
$$
For each $d \in D$, the set $\set{y \in L}{[y,x] = d(x)}$ is an affine space 
over the vector subspace $\C_S(x)$, and hence 
$$
\dim \C_L(x) = \dim \C_S(x) + \dim D = \rk L + \dim D .
$$

Any element $d \in L$, $d \notin S$ may be ``fixed'' by an inner derivation of 
$S$ to represent an outer derivation, so we may assume $d \in D$ for an 
appropriate decomposition (\ref{eq-sd}). Then, since $[D,D] = 0$, we have 
$\C_L(d) = \Ker(d) \dotplus D$. Further, $\Ker(d)$ is a subalgebra of $S$, and 
$\Ker(d) \dotplus Kd$ is a subalgebra of $L$ with a central element $d$, thus 
$\Ker(d) \dotplus Kd$ cannot be almost simple, and hence is abelian. Now, 
$\N_L(\Ker(d))$ is a subalgebra of $S$ having the abelian ideal $\Ker(d)$. Note
that according to \S \ref{s-not}.\ref{ss-as}, $\Ker(d) \ne 0$, and hence 
$\N_L(\Ker(d))$ is abelian.

We also have:
\begin{multline*}
[d(\N_L(\Ker(d))), \Ker(d)] \subseteq 
 d([\N_L(\Ker(d)), \Ker(d)]) + [\N_L(\Ker(d)), d(\Ker(d))]
\\ \subseteq d(\Ker(d)) = 0 \subset \Ker(d) ,
\end{multline*}
what shows that $\N_L(\Ker(d))$ is invariant under $d$. Thus 
$\N_L(\Ker(d)) \dotplus Kd$ is a subalgebra of $L$ with an abelian subalgebra 
$\N_L(\Ker(d))$ of codimension $1$. Clearly, $\N_L(\Ker(d)) \dotplus Kd$ cannot
be almost simple, and hence is abelian. This means that 
$\N_L(\Ker(d)) \subseteq \Ker(d)$, and hence $\N_L(\Ker(d)) = \Ker(d)$. Thus
$\Ker(d)$ is a Cartan subalgebra of $S$, and hence $\dim \Ker(d) = \rk L$, and
$\dim \C_L(d) = \rk L + \dim D$.

We see that centralizers of all nonzero elements in $L$ have the same dimension
$\rk L + \dim D$. Since $L$ is anisotropic, for any $x \in L$ the space $L^0(x)$
coincides with the centralizer $\C_L(x)$, and therefore any nonzero element of
$L$ is regular.
\end{proof}

\section{Real regular Lie algebras}\label{s-real}

If the ground field is algebraically closed, the only regular Lie algebras are
nilpotent ones. This follows from Lemma \ref{lemma-3}, Theorem \ref{th-subalg}, 
and an obvious fact that over an algebraically closed field any non-nilpotent 
Lie algebra contains the two-dimensional nonabelian subalgebra (consider an
eigenvector of a non-nilpotent adjoint map).

A description of regular Lie algebras over $\mathbb R$, the field of real 
numbers, is almost as easy and readily follows from the developed structure 
theory. Recall that $\su_2(\mathbb R)$ denotes the simple $3$-dimensional 
compact Lie algebra, a real form of $\Sl_2(\mathbb C)$.

\begin{theorem}\label{th-r} (The ground field is $\mathbb R$).
A Lie algebra is regular if an only if it is either nilpotent, or isomorphic to
$\su_2(\mathbb R)$.
\end{theorem}

\begin{proof}
The ``if'' part follows from Lemma \ref{lemma-1} and the fact that 
$\su_2(\mathbb R)$ is regular. The latter can be either verified directly, or 
follows from the fact that $\su_2(\mathbb R)$ is isomorphic to 
$Q^{(-)}/\mathbb R 1$, where $Q$ is the quaternionic real division algebra (of
degree $2$), and Corollary \ref{cor-d} below.

The ``only if'' part: By Lemma \ref{lemma-3}, Lemma \ref{lemma-almost}, and
Lemma \ref{lemma-nn}, a regular Lie algebra over a field of characteristic zero
is either nilpotent, or simple anisotropic. Now the structure theory of real
simple Lie algebras readily implies that a real simple anisotropic Lie algebra 
is compact (see, for example, \cite[Theorem 2]{singer} or 
\cite[Theorem 2]{sugiura}). Further, the structure of the compact form of a 
complex simple Lie algebra (as expounded, for example, in 
\cite[Chapter IX, \S 3, Proof of Proposition 2]{bourbaki} or 
\cite[Chapter IV, \S 7, Proof of Theorem 10]{jacobson}) implies that 
algebras of rank $>1$ are not regular. Indeed, if 
$$
\mathfrak g = \mathfrak h \dotplus \bigoplus_{\alpha \in R} \mathbb C e_\alpha
$$
is a Cartan decomposition of a complex simple Lie algebra $\mathfrak g$ with 
Cartan subalgebra $\mathfrak h$ and the root system $R$, then the compact form 
of $\mathfrak g$ can be constructed as 
$$
\mathfrak g_c = \mathfrak h_c \dotplus \bigoplus_{\alpha \in R_+} 
(\mathbb R u_\alpha \dotplus \mathbb R v_\alpha) ,
$$
where 
$$
\mathfrak h_c = 
i \set{h \in \mathfrak h}{\alpha(h) \in \mathbb R \text{ for any } \alpha \in R}
$$
is the Cartan subalgebra of $\mathfrak g_c$, 
$u_\alpha = e_\alpha + e_{-\alpha}$, $v_\alpha = i(e_\alpha - e_{-\alpha})$, and
$R_+$ is the set of positive roots. 

Then if $\rk \mathfrak g_c = \rk \mathfrak g > 1$, the algebra 
$\mathfrak g_c$ has a lot of non-regular subalgebras: for example, 
$$
\mathfrak h_c \dotplus \mathbb R u_\alpha \dotplus \mathbb R v_\alpha
$$
for any fixed $\alpha \in R_+$, which is isomorphic to a split central extension
of $\su_2(\mathbb R)$, or $\bigoplus_{\alpha \in R_+} \mathbb R u_\alpha$. Thus
we are left with the compact Lie algebra of rank one, which is isomorphic to 
$\su_2(\mathbb R)$.
\end{proof}

\section{
Minimal non-abelian, minimal non-regular, and minimal non-commutative algebras
}\label{s-min}

In this section we discuss some ramifications and consequences of 
Theorem~\ref{th-subalg}, and connection with the class of minimal nonabelian Lie
algebras.

An alternative, more concrete, way to handle the almost simple case in the
proof of implication (iii) $\Rightarrow$ (i) of Theorem \ref{th-subalg}, would 
be the following. Recall that the task there is boiled down to this: if $S$ is a
regular simple Lie algebra, prove that its derivation algebra $\Der(S)$ is also
regular. By \cite[Chapter X, \S 1, Theorem 5]{jacobson}, for any $d \in \Der(S)$
and $\omega \in \Omega(S)$, there is a derivation $\delta \in \Der_K(\Omega(S))$ such that $[d,\omega] = \delta(\omega)$. But 
since the ground field $K$ is perfect, the finite field extension 
$K \subset \Omega(S)$ is separable, and $\Der_K(\Omega(S)) = 0$. Therefore, 
$\Der(S)$ carries a structure of a vector space over $\Omega(S)$, and 
$\Der_{\Omega(S)}(S) = \Der(S)_{\Omega(S)}$. Since the fact whether the algebra is regular or not does not change under restriction over a 
larger field, to prove regularity of $\Der(S)$, we can pass to the restriction 
over $\Omega(S)$, and assume that $S$ is central simple. 

According to \S \ref{s-not}.\ref{ss-as}, we have $\Der(S) \not\simeq S$ only
if characteristic of the ground field is $p>0$, and $S$ is a form of the Lie 
algebra $\psl_{kp}(\overline K)$ for some $k$. According to 
\S \ref{s-not}.\ref{ss-forms}, $S$ is isomorphic either to an algebra 
$[A^{(-)},A^{(-)}]/K1$, where $A$ is a central simple associative algebra, or to
an algebra $[\sym^-(A,J),\sym^-(A,J)]/K1$, where $A$ is a simple associative 
algebra whose center is a quadratic extension of $K$, and $J$ is involution on 
$A$ of the second kind; in both cases the degree of $A$ over its center is equal
to $kp$.

In these two cases, we have $\Der(S) \simeq A^{(-)}/K1$, and 
$\Der(S) \simeq \sym^-(A,J)/K1$, respectively. In the first case, it is 
straightforward to see that if $[A^{(-)},A^{(-)}]/K1$ is anisotropic, then $A$ 
is a division algebra. Then by a reasoning very similar to the proof of 
\cite[Theorem 6]{mna} (reproduced as Theorem \ref{min-noncomm} below; the only 
difference is that we should deal with the algebra $[A^{(-)},A^{(-)}]/K1$ 
instead of $A^{(-)}/K1$), $A$ is minimal noncommutative. Then by the same 
theorem $A^{(-)}/K1$ is regular, and we are done. The second case, involving 
$\sym^-(A,J)$ with involution of the second kind, is a bit more complicated, but
could be treated similarly.

\begin{corollary}[to Theorem \ref{th-subalg}] 
(The ground field $K$ is perfect of characteristic $\ne 2,3$).
A simple minimal non-abelian Lie algebra is regular.
\end{corollary}

Over some fields the converse is true: trivially, over algebraically closed 
fields, where both classes of simple algebras are empty; or over $\mathbb R$, 
where both class of algebras consists of the single $3$-dimensional nonsplit 
simple algebra, see Theorem \ref{th-r}. Generally, the converse is not true: 
take, for example, a division algebra $D$ such that the Lie algebra 
$[D^{(-)},D^{(-)}]$ is minimal nonabelian over the center $\Z(D)$ (for example,
take $D$ of prime degree over $\Z(D)$, see \cite[Theorem 5.2]{gein} or 
\cite[Theorem 9.2]{gein-book}). Then $[D^{(-)},D^{(-)}]$ is regular over 
$\Z(D)$, and hence is regular over $K$. Assume now that there are ``enough'' 
intermediate fields $K \subset F \subset \Z(D)$. Then we can choose $F$ such 
that the centralizer $\C_D(F)$ is noncommutative, and hence the $K$-algebra $[D^{(-)},D^{(-)}]$ 
contains a proper nonabelian Lie subalgebra $[\C_D(F),\C_D(F)]$.

Of course, the Lie algebra $[D^{(-)},D^{(-)}]$ in this example is not central. 
We believe that there exist \emph{central} simple regular Lie algebras which are
not minimal nonabelian, but fail to provide an example. Probably, it can be 
found by refining some arguments in \cite{gein}. Note that according to 
\cite[Theorem 1]{muller}, over a local field of characteristic zero, no such 
central simple Lie algebras exist.

Now let us discuss the restrictions on $L$ in the condition of 
Theorem~\ref{th-subalg}. 
Both restrictions -- non-solvability and triviality of
the center -- are essential. Examples of non-regular Lie algebras which are 
either solvable, or with a nontrivial center, all whose proper subalgebras are 
regular, are provided by the following description of minimal non-regular 
algebras.

\begin{theorem}\label{th-minnonreg}
(The ground field $K$ is perfect of characteristic $p\ne 2,3$).
A Lie algebra $L$ is minimal non-regular if and only if one of the following 
holds:
\begin{enumerate}[\upshape(i)]
\item\label{i-solv} 
$L$ is solvable minimal non-nilpotent;

\item\label{i-ext}
$L$ is a split one-dimensional central extension of a minimal non-abelian simple
Lie algebra;

\item\label{i-mna}
$p>0$, $L$ is not regular, is minimal non-abelian, and is a nonsplit central 
extension of a simple minimal non-abelian Lie algebra which over its centroid is
isomorphic to a form of an algebra of the kind $\psl_{kp}(\overline K)$.
\end{enumerate}
\end{theorem}

\begin{proof}
The ``if'' part is obvious, so let us prove the ``only if'' part. If $L$ is 
solvable, then by Lemma \ref{lemma-3}, all its proper subalgebras are nilpotent,
and we are in part (\ref{i-solv}). So we may assume that $L$ is not solvable.
By Theorem~\ref{th-subalg}, $Z(L) \ne 0$.

\emph{Case 1}. $L$ is minimal non-nilpotent. By \cite[Theorem 8.8]{gein-book}, 
$L$ is minimal non-abelian, and $L/Z(L)$ is simple. If $\overline A$ is a proper
subalgebra of $L/Z(L)$, then the preimage of $\overline A$ in $L$ is a proper 
subalgebra of $L$, and hence is abelian. This shows that $L/\Z(L)$ is minimal 
non-abelian. By \cite[Lemma 7]{premet-cartan}, or by 
\cite[Theorem 4.1]{farn-anis}, any simple minimal non-abelian Lie algebra is 
anisotropic, and then by Theorem A, $L/Z(L)$ is a form of a classical Lie 
algebra.

If the central extension $0 \to Z(L) \to L \to L/Z(L) \to 0$ splits, then 
$L/Z(L)$ is a proper subalgebra of $L$, a contradiction. Therefore, this central
extension does not split, and the second degree cohomology $\Homol^2(L/Z(L),K)$
does not vanish. As cohomology is preserved under field extensions, 
$\Homol^2(L/Z(L) \otimes_K \overline K, \overline K) \ne 0$. According to 
\cite[Theorem 3.1]{B1}, among classical Lie algebras, the second degree 
cohomology with trivial coefficients vanishes except for the case of the algebra
$\psl_{pk}(K)$, in which case $\Homol^2(\psl_{pk}(K),K)$ is one-dimensional.
Therefore, 
$(L \otimes_K \overline K) / Z(L \otimes_K \overline K) \simeq 
L/Z(L) \otimes_K \overline K$ 
is isomorphic to the direct sum of several copies of $\psl_{pk}(\overline K)$ 
for some $k$, and hence $L/Z(L)$, being restricted over its centroid, is 
isomorphic to a form of $\psl_{pk}(\overline K)$. We are in part (\ref{i-mna}).

\emph{Case 2}. $L$ contains a proper non-nilpotent subalgebra $S$. For any 
nonzero $z \in Z(L)$, $S + Kz$ is a subalgebra of $L$ with a central element 
$z$. If it is a proper subalgebra, then it is nilpotent, a contradiction. Hence
$L = S \dotplus Kz$, and $Z(L) = Kz$ is one-dimensional. If $A$ is a proper 
subalgebra of $S$, then $A \dotplus Kz$ is a proper subalgebra of $L$ with a
nonzero central element $z$, and hence $A \dotplus Kz$ is nilpotent. Therefore,
$S$ is minimal non-nilpotent. Since $S$ is regular, by Theorem~\ref{th-as} it is
almost simple, and since $S$ is minimal non-nilpotent, it is simple. Now by 
\cite[Lemma 7]{premet-cartan}, or by \cite[Theorem 8.8, part 3]{gein-book}, $S$
is minimal non-abelian, and we are in part (\ref{i-ext}).
\end{proof}

Further comments about the conditions on Lie algebras arising in this theorem 
are in order. Solvable minimal non-nilpotent Lie algebras, i.e. algebras in part
(\ref{i-solv}), are described in \cite[Theorem 8.5]{gein-book}, 
\cite{stitzinger}, and \cite{towers-laa}. All of them are isomorphic to a 
semidirect sum of the kind $L = N \dotplus Kx$, where $N$ is nilpotent (of index
$3$ if the ground field is perfect), and $\ad x$ acts on $N$ in a specific way.

Concerning part (\ref{i-mna}): if $L/Z(L)$ is central, then it is a form of 
$\psl_{pk}(\overline K)$ for some $k$. According to \cite[Theorem 3.1]{B1},
any nonsplit central extension of $\psl_{pk}(\overline K)$ is one-dimensional,
and is isomorphic to $\Sl_{pk}(\overline K)$. According to 
\S \ref{s-not}.\ref{ss-forms}, in this case $L$ is isomorphic either to an 
algebra $[A^{(-)},A^{(-)}]$, where $A$ is a central simple associative algebra 
$A$ of degree $kp$, or to $[\sym^-(A,J),\sym^-(A,J)]$, where $A$ is a simple 
associative algebra $A$ with involution $J$ of the second kind. It is a natural
and interesting question when algebras of this kind, and, generally, central 
simple forms of simple classical Lie algebras of all types, and close to them 
algebras, are regular. For types B--D this amounts to studying the Lie algebras 
$\sym^-(A,J)$, where $A$ is a central simple associative algebra with involution
$J$ of the first kind, and is more or less straightforward. For exceptional 
types, the situation is more difficult. We hope to treat this question 
elsewhere.

A first step in this direction is

\begin{theorem}\label{min-noncomm}
Let $D$ be a central associative division algebra. Then the Lie algebra 
$D^{(-)}/K1$ is regular if and only if $D$ is a minimal non-commutative algebra.
\end{theorem}

\begin{proof}
This is \cite[Theorem 6]{mna}. The proof given there, based on the double 
centralizer theorem, is characteristic-free.
\end{proof}

\begin{corollary}\label{cor-d}
If $D$ is a central associative division algebra of prime degree, then the Lie
algebra $D^{(-)}/K1$ is regular.
\end{corollary}

\begin{proof}
According to \cite[Corollary 1.2]{gein}, $D$ is minimal non-commutative, and by
Theorem \ref{min-noncomm}, $D^{(-)}/K1$ is regular.
\end{proof}

\section{Commutators}\label{s-ass}

In 1936 Shoda proved that any $n \times n$ matrix of trace zero over a field of 
characteristic zero can be represented as a commutator of two matrices 
(\cite[Satz 3]{shoda}). Two decades later Albert and Muckenhoupt, 
\cite{albert-muck}, extended this result to matrices over an arbitrary field 
$K$. As matrices of trace zero form the Lie algebra $\Sl_n(K)$ under the commutator, this result 
can be formulated as follows: any element of $\Sl_n(K)$ can be represented as a
commutator. Brown, \cite{brown}, extended this result to all finite-dimensional
simple Lie algebras of classical type.

It is natural to ask the same question for arbitrary finite-dimensional simple
Lie algebras, in particular, for forms of the Lie algebra $\Sl_n(K)$. As 
indicated in \S \ref{s-not}.\ref{ss-forms}, such forms are described in terms of
simple associative algebras. In particular, if $A$ is a central simple 
associative algebra of degree $n$ over a field $K$, then 
$A \otimes_K \overline K \simeq M_n(\overline K)$, the full matrix algebra over
the algebraic closure $\overline K$, and we can define the \emph{reduced trace}
$\tr: A \to K$ by the formula $\tr(x) = \tr(x \otimes 1)$ for $x \in A$, where 
at the right-hand side of the equality stands the usual trace in 
$M_n(\overline K)$. Then 
$$
\set{x \in A}{\tr(x) = 0} = [A,A] ,
$$
so any element of reduced trace zero in $A$ is a \emph{sum} of commutators of
elements of $A$. Then one can ask whether any element of reduced trace zero in 
$A$ is \emph{a} commutator. This question was addressed by Amitsur and Rowen, 
\cite{trace-zero}. It was proved there that if $A$ is a matrix algebra of order
$n>1$ over a division algebra, then any noncentral element of reduced trace zero
(in particular, any element of reduced trace zero if characteristic of the 
ground field $K$ is zero), is a commutator, and further results were obtained 
for some particular cases of division algebras. However, in general the question
remains open, and, as reported in \cite[\S 2.1.1(v)]{kkp}, it is believed that 
in general the answer is negative.

Our contribution to this problem is the following

\begin{conjecture}
In a non-nilpotent regular Lie algebra having a nondegenerate symmetric 
invariant bilinear form, any element of a commutant is a commutator.
\end{conjecture}

This conjecture would imply that in any central minimal noncommutative division
algebra $D$ whose degree is not divided by the characteristic of the ground 
field $p$ (and, in particular, a central division algebra of prime degree 
different from $p$), any element of reduced trace zero is a commutator. Indeed:
by Theorem~\ref{min-noncomm}, $D^{(-)}/K1$ is regular; since the degree of $D$ 
is not divided by $p$, $[D^{(-)}, D^{(-)}] \simeq D^{(-)}/K1$; and the Lie 
algebra $[D^{(-)}, D^{(-)}]$, being a form of $\Sl_n(K)$, has a nondegenerate symmetric
invariant bilinear form.

Over a perfect field of characteristic $\ne 2,3$, the conjecture is, 
essentially, about simple forms of classical Lie algebras. Indeed, by 
Theorem~\ref{th-as}, we are talking in this case about almost simple Lie 
algebras: $S \subseteq L \subseteq \Der(S)$, where $S$ is a simple form of a 
classical Lie algebra. By the same reasoning as at the beginning of 
\S \ref{s-min}, we can restrict this situation over the centroid $\Omega(S)$ of
$S$: $S_{\Omega(S)} \subseteq L_{\Omega(S)} \subseteq \Der(S_{\Omega(S)})$, and 
$L_{\Omega(S)}$ is not simple only in the case where $S_{\Omega(S)}$ is a form 
of $\psl_{kp}(\overline K)$ for some integer $k$, and 
$L_{\Omega(S)} = \Der(S_{\Omega(S)})$ is a form of $\pgl_{kp}(\overline K)$. 
But neither $\psl_{kp}(\overline K)$, nor $\pgl_{kp}(\overline K)$ have a 
nondegenerate invariant symmetric bilinear form, hence the same is true for 
$\Omega(S)$-algebras $S_{\Omega(S)}$ and $L_{\Omega(S)}$, and then for 
$K$-algebras $S$ and $L$, so this case is excluded.

However, a possible attack on this conjecture is not to utilize specificity of
forms of classical algebras, but general properties of Cartan subalgebras in 
regular Lie algebras, following the approach of \cite{brown} (where it is proved
that any element in a classical simple Lie algebra is a commutator), and also of
\cite[Appendix 3]{hofmann-morr}, \cite{akhiezer}, and \cite{m-nahlus} (where the
same is proved for the most real simple Lie algebras, including the compact 
ones). In all these treatises, it is proved -- via judicious manipulation with 
the automorphism group in the classical (split) case, and with the corresponding
Lie group in the real case -- that any element in a simple Lie algebra 
$\mathfrak g$ under consideration is conjugate to an element lying in the 
Fitting $1$-component with respect to a suitable Cartan subalgebra 
$\mathfrak h$. But $\mathfrak h = C_{\mathfrak g}(x)$ for some element 
$x \in \mathfrak h$, and then $\mathfrak g^1(\mathfrak h) = [\mathfrak g,x]$. 
Consequently, any element in the Fitting $1$-component with respect to a Cartan
subalgebra is a commutator, and thus any element of $\mathfrak g$ is a 
commutator (with the corresponding element $x$).

In the regular case, instead of manipulating with the automorphism group -- 
which, in the case of an arbitrary field and an arbitrary form, appears to be 
problematic -- we may try to use the flexibility of the choice of $\mathfrak h$
and $x$ (in fact, any nonzero $x$ will do!), and the presence of nondegenerate symmetric 
invariant bilinear form.

A first step might be provided by the following

\begin{lemma}\label{lemma-cart}
Let $L$ be a Lie algebra with a nondegenerate symmetric invariant bilinear form,
and $X$ a set of almost commuting elements of $L$. Then $L^0(X)^\perp = L^1(X)$.
\end{lemma}

\begin{proof}
Let $\aform$ be a nondegenerate symmetric invariant bilinear form on $L$. For 
any $x,y,z \in L$, and any $n$, we have
$$
\langle y,(\ad x)^n(z) \rangle = (-1)^n \langle (\ad x)^n(y),z \rangle .
$$

If $y \in L^0(X)$, this equality implies that $y \in L^1(x)^\perp$ for any
$x \in X$, and hence $y \in L^1(X)^\perp$. Therefore, 
$L^0(X) \subseteq L^1(X)^\perp$. Due to the Fitting decomposition, we have 
$\dim L^0(X) + \dim L^1(X) = \dim L$. On the other hand, since $\aform$ is 
nondegenerate, $\dim L^0(X) + \dim L^0(X)^\perp = \dim L$. Consequently, 
$\dim L^0(X)^\perp = \dim L^1(X)$, and $L^0(X)^\perp = L^1(X)$.
\end{proof}

Lemma \ref{lemma-cart} is just a slight generalization of 
\cite[Lemma 2.1]{akhiezer} and \cite[Remark 5.1]{m-nahlus}.

The suggested approach is illustrated by the following elementary 

\begin{lemma}\label{th-mna}
In a simple regular Lie algebra of rank $1$, having a nondegenerate symmetric 
invariant bilinear form, any element is a commutator.
\end{lemma}

\begin{proof}
Let $L$ be a simple regular Lie algebra of rank $1$, with a nondegenerate 
symmetric invariant bilinear form $\aform$, and $x$ a nonzero element of $L$. 
Take a nonzero $y \in L$ such that $\langle x,y \rangle = 0$. Then $Ky$ is a Cartan
subalgebra of $L$, and by Lemma~\ref{lemma-cart}, $x \in (Ky)^\perp = [L,y]$. 
\end{proof}

\begin{corollary}
In a quaternion division algebra, any element of reduced trace zero is a 
commutator.
\end{corollary}

\begin{proof}
If $Q$ is a quaternion division algebra, then elements of reduced trace zero
form a Lie algebra $[Q^{(-)},Q^{(-)}]$, which satisfies the conditions of 
Lemma~\ref{th-mna}.
\end{proof}

This elementary result is, of course, not new, and can be checked by simple
direct computations. In \cite[Corollary 0.9]{trace-zero} another elementary 
proof, based on the existence of inseparable elements, is given.

In the general case, the task is to ensure, for any nonzero $x \in L$, the 
existence of an element whose centralizer lies in $(Kx)^\perp$, a subspace of
$L$ of codimension $1$. This seems to be highly plausible, but the proof is 
elusive so far.

\section*{Acknowledgement}

Thanks are due to Alexander Gein and Vicente Varea for supplying me with a
hard-to-find item \cite{gein-book}.

\end{document}